\documentclass{elsarticle}

\usepackage[utf8]{inputenc}
\usepackage[english]{babel}
\usepackage{amsmath,amsthm}
\usepackage{amsfonts}
\usepackage{amssymb}

\usepackage{enumerate}  
\usepackage{hyperref}  
\hypersetup{colorlinks=true,linktoc=none, linkcolor=blue,citecolor=red,urlcolor=blue}
\usepackage{etoolbox}

\newcommand{\e}{\mathrm{e}}  
\newcommand{\D}{\mathcal{D}}  
  
\newcommand{\R}{\mathbb{R}}  
 
\theoremstyle{plain}
\newtheorem{theorem}{Theorem}
\newtheorem{proposition}{Proposition}
\newtheorem{lemma}{Lemma}
\newtheorem{claim}{Claim}
\theoremstyle{definition}
\theoremstyle{remark}
\newtheorem{remark}{Remark}

\numberwithin{equation}{section}

\makeatletter
\def\ps@pprintTitle{%
 \let\@oddhead\@empty
 \let\@evenhead\@empty
 \def\@oddfoot{}%
 \let\@evenfoot\@oddfoot}
\makeatother

\begin{document}

\title{Regular solutions to a supercritical elliptic problem in exterior domains\footnote{NOTICE: this is the author's version of a work that was accepted for publication. Changes resulting from the publishing process, such as peer review, editing, corrections, structural formatting, and other quality control mechanisms may not be reflected in this document. Changes may have been made to this work since it was submitted for publication. A definitive version was subsequently published in J. Differential Equ., \textbf{255} (2013), no. 4, 701--727, \url{http://dx.doi.org/10.1016/j.jde.2013.04.024}}}

\author[JD]{Juan D\'avila} 
\address[JD]{Departamento de Ingenier\'{\i}a Matem\'atica and CMM, Universidad
de Chile, Casilla 170 Correo 3, Santiago, Chile}
\ead{jdavila@dim.uchile.cl}

\author[LL]{Luis F. L\'{o}pez}
\address[LL]{Departamento de Ingenier\'{\i}a Matem\'atica, Universidad
de Chile, Casilla 170 Correo 3, Santiago, Chile}
\ead{llopez@dim.uchile.cl}


\begin{abstract}
We consider the supercritical elliptic problem $-\Delta u=\lambda\e^u$, $\lambda>0$, in an exterior domain $\Omega=\R^N\setminus\D$ under zero Dirichlet condition, where $\D$ is smooth and bounded in $\R^N, N\geq 3$. We prove that, for $\lambda$ small, this problem admits infinitely many regular solutions.  
\end{abstract}


\maketitle


\section{Introduction and main results}
Let $\D$ be a bounded, smooth domain in $\R^N$, $N\geq 3$. In this paper we consider the problem of finding classical solutions of
\begin{align}
-\Delta u &=\lambda\e^u \quad \text{in}\ \mathbb{R}^N\setminus\overline{\D},\label{1}\\
u&=0 \qquad\text{on}\ \partial \D,
\label{2}
\end{align}
where $\lambda>0$ is a parameter.

V\'eron and Matano \cite[Theorem~3.1]{Ver1992} considered Eq.~\eqref{1} in dimension 3 with $\D=B_1(0)$, the unit ball, and a non-homogeneous boundary condition $u = \phi$ on $\partial B_1(0)$. They proved that if $\phi$ is close enough to $2 w - \log(\lambda/2)$ where $w$ is a smooth solution of 
\begin{align}
\label{eqs2}
\Delta_{S^2} w + \e^{2w}-1=0
\end{align}
on the sphere $S^2$, then there is a solution with the asymptotic behavior
$$
u(x) = -2\log|x|-\log(\lambda/2) + 2 w(x/|x|) + o (1)
$$
as $|x|\to+\infty$. This result is based on the understanding of the solutions of \eqref{eqs2} obtained in \cite{CY1987} (see also \cite{BV1991}).

In all dimensions $N\geq 3$ and still with  $\D=B_1(0)$ we can also describe many radial solutions of \eqref{1}--\eqref{2}. To fix ideas, we denote by $U$ the unique radial solution of 
\begin{align}
-\Delta U &=\lambda_0\e^U \quad  \text{in } \R^N,\label{6}\\
U(0)&=0, \label{7}
\end{align}
where $\lambda_0:=2(N-2)$.
This solution can be constructed by solving the initial value problem \eqref{6} with $U(0)=U'(0)=0$, and can be proved to be defined for all $r>0$.
For any $\alpha>0$ the function 
\begin{equation}
U_\alpha=U(\alpha r)+2\log\alpha \label{8}
\end{equation} 
also satisfies \eqref{6}.
Note that 
for $\lambda>0$, $u = U_\alpha - \log(\frac{\lambda}{\lambda_0})$
satisfies 
$$
-\Delta u= \lambda \e^u\quad\text{in}\ \R^N .
$$
The boundary condition \eqref{2} is fulfilled if  $0 = U(\alpha) + 2\log\alpha- \log(\frac{\lambda}{\lambda_0})$. 

Regarding  $\alpha>0$ as a parameter we  find a family of solutions of \eqref{1}--\eqref{2} of the form $\lambda_\alpha = \lambda_0 \alpha^2 \e^{U(\alpha)}$ and $u_\alpha(r) = U_\alpha(r) - U_\alpha(1)$.
As $\alpha\to 0$ we see that $\lambda_\alpha\to 0$, while $\lambda_\alpha\to\lambda_0$  as $\alpha \to +\infty$, which follows from the asymptotic behavior of $U(r)= -2\log(r)+o(1)$ as $r\to+\infty$.
Let us point out that the family of solutions
$(\lambda_\alpha,u_\alpha)$ with $\alpha >0$ also describes all  classical solutions of the  problem
$$
-\Delta u = \lambda \e^u \quad \text{in } B_1(0), \quad
u=0\quad\text{on } \partial B_1(0),
$$
with $\lambda >0$,
which was studied in dimension 3 in \cite{Gel1959} and later in all higher dimensions in \cite{JL1972}.

Still in the case $\D=B_1(0)$ one can see that the set of classical solutions of \eqref{1}--\eqref{2} is much richer than the family given by 
$(\lambda_\alpha,u_\alpha)$ with $\alpha >0$. To see this it is convenient
to work with Emden-Fowler change of variables 
\begin{equation}
\label{3c}
v(s)=u(r),\ r=\e^s, 
\end{equation}
which transform \eqref{1} into 
\begin{equation}\label{4z}
v''+(N-2)v'=-\lambda\e^{v+2s}\quad\text{in}\ \R,
\end{equation} 
and then define 
\begin{equation}\label{4a}
v_1=\lambda\e^{v+2s}, \quad 
v_2=v'.	
\end{equation}
This transforms \eqref{4z} into the autonomous system
\begin{equation}\label{4bz}
\begin{pmatrix}
v_1\\v_2
\end{pmatrix}'
=
\begin{pmatrix}
v_1(v_2+2)\\
-v_1-(N-2)v_2
\end{pmatrix}.
\end{equation}  
This system has two stationary points: $(0,0)$ and $(2(N-2),-2)$, the first being a saddle point and the second an spiral or an asymptotically stable node depending on the dimension. 

The solution $U$ of \eqref{6}--\eqref{7} corresponds to a heteroclinic orbit which connects the equilibria $(0,0)$ and $(2(N-2),-2)$ in the phase plane $(v_1,v_2)$. Take any point $P= (\lambda,\beta)$ in this orbit. Then for any $\tilde P = (\lambda,\tilde \beta)$ sufficiently close to $P$ the solution of \eqref{4bz} with initial condition $\tilde P$ at time $s=0$ will be defined for all positive $s$ and will converge to $(2(N-2),-2)$ as $s\to+\infty$. Under the previous change of variables this will be a solution of \eqref{1}--\eqref{2} associated to the same parameter $\lambda$. Note that $\lambda = \lambda_\alpha$ for some $\alpha>0$. The previous discussion shows that together with the {\em special} solution $(\lambda_\alpha,u_\alpha)$ there is a continuum of other solutions of \eqref{1}--\eqref{2} with the same $\lambda_\alpha$, and all share the behavior $u(r) = - 2\log(r)-\log(\frac{\lambda}{2(N-2)}) + o(1) $ as $r\to+\infty$.

The purpose of this paper is to show that  part of the family of solutions described in the preceding paragraph in the radial setting still exists for general exterior domains $\mathbb{R}^N\setminus\overline{\D}$. 
For the precise statement of our result, we need to distinguish between the cases $N\geq 4$ and $N=3$. 

\begin{theorem}\label{4d}
Assume $N\geq 4$. Let $\alpha>0$ and $\xi\in\R^N$. Then for sufficiently small $\lambda>0$ there is a solution $u_\lambda$  to problem \eqref{1}--\eqref{2} such that 
\[u_\lambda\to U_\alpha(\xi)(1-\varphi_0(x))\quad\text{as}\ \lambda\to 0^+,\]
uniformly on bounded subsets of $\R^N\setminus\D$, where $\varphi_0$ is the Newtonian potential of the layer $\partial\D$ (see \eqref{9a}), and has the asymptotic behavior \[u_\lambda(x)=-2\log|x|-\log\left(\frac{\lambda}{2(N-2)}\right)+O(|x|^{-\beta})\quad\text{as}\ |x|\to+\infty,\] where $\beta$ is a positive number (see \eqref{11d}). 
\end{theorem}

The  analysis of the radial case suggests looking for a solution $u_\lambda$ close to a rescaled and translated form of $U_\alpha$. It turns out that  $U_\alpha(\xi+\sqrt{\lambda/\lambda_0}x)$, where $\xi\in\R^N$ is fixed arbitrarily, is a good approximation. This leads us to construct an inverse of the linearized operator $-\Delta u+\lambda_0\e^{U_\alpha}$ in $\R^N\setminus(\xi+\sqrt{\lambda/\lambda_0}\D)$. This set approaches  $\R^N$ as $\lambda\to 0^+$, so such an inverse is constructed as a small perturbation of an inverse of this operator in entire space. This inverse indeed exists for $N\geq 4$ and adding a lower order correction to the initial approximation yields the desired solution. 
In dimension 3, however, the linearized operator is not surjective, having a range orthogonal to the generator of translations.
Thus for $N=3$ we find a family of solutions provided $\xi$ is adjusted properly. This explains the next result. 

\begin{theorem}\label{42a}
Let $\alpha>0$. Then there exist $\Lambda,Z>0$ such that for $0<\lambda<\Lambda$ there are $\xi_\lambda\in\R^3$ with $|\xi_\lambda|<Z$ and a solution $u_\lambda$ to problem \eqref{1}--\eqref{2} such that 
\[
u_{\lambda}(x)-U_\alpha(\xi_\lambda)(1-\varphi_0(x))\to 0\quad\text{as}\ \lambda\to 0^+,
\]
uniformly on bounded subsets of $\R^3\setminus\D$ (see \eqref{9a} for $\varphi_0$). Moreover, $u_\lambda$ has the behavior
\[
u_\lambda(x)=-2\log|x|-\log\left(\frac{\lambda}{2(N-2)}\right)+O(|x|^{-\beta})\quad\text{as}\ |x|\to+\infty,
\]
where $\beta\in(0,1/2)$. 
\end{theorem}

In summary, the difference between the cases $N=3$ and $N\geq 4$ is that in the former case, the solutions found constitute a one-parameter family only dependent on $\alpha>0$, while in the latter case is an $(N+1)$-dimensional family depending on $\alpha>0$ and $\xi$.

Theorem~\ref{42a} is similar to the one of V\'eron and Matano \cite[Theorem~3.1]{Ver1992} where we choose to work with the solution $w=0$ of \eqref{eqs2}, but we obtain existence for a general bounded smooth domain $\D$.

Similar phenomenon to the one presented in this work for the exponential nonlinearity was detected for a supercritical equation in exterior domains in \cite{DPM2007,DPMW2008} and for a supercritical Schr\"odinger equation in entire space, with a rapidly decaying potential in \cite{DPMW2007}. 


\section{Preliminaries}
Let us make some additional comments on \eqref{6}--\eqref{7}  and system \eqref{4bz}. As we mentioned, the solution $U$ to \eqref{6}--\eqref{7} can be obtained using the Picard fixed point theorem applied to the equivalent integral equation
\[
u(r)=-\lambda\int_0^r\int_0^s\left(\dfrac{t}{s}\right)^{N-1}\e^{u(t)}\,\mathrm{d}t\,\mathrm{d}s
\]
on a maximal interval $(0,T)$. We can show that in fact $T=+\infty$ by observing that $v$, defined in \eqref{3c}, remains bounded in $(-\infty,\log T)$. Indeed, note that the  Lyapunov function
$$\mathcal{L}(v)=\dfrac{(v'+2)^2}{2}+\lambda\e^{v+2s}-2(N-2)(v+2s)$$
decreases.

\subsection{Phase plane analysis}
Recall that the system \eqref{4bz} has two stationary points $(0,0)$ and $(\lambda_0,-2)$, where $\lambda_0=2(N-2)$. If we linearize around the second point we have the associated eigenvalues
\begin{equation}\label{5}
\mu_\pm=-\dfrac{N-2}{2}\pm\dfrac{1}{2}\sqrt{(N-2)(N-10)}.
\end{equation}
If $3\leq N\leq 9$, \eqref{5} gives complex values and $(\lambda_0,-2)$ is a spiral point. 
If $N\geq 10$, \eqref{5} gives negative eigenvalues and $(\lambda_0,-2)$ is an asymptotically stable node. 

We can get an expression of $U$ and $U'$ in terms of $v_1$ and $v_2$, namely
\begin{equation}\label{7a}
\e^{U(r)}=\frac{v_1(\log r)}{\lambda_0}r^{-2},\quad U'(r)=r^{-1}v_2(\log r),
\end{equation}
which implies that $U(r) = -2\log(r) + o(1)$ as $r\to+\infty$. This behavior is actually common to all the radial solutions of \eqref{6}.

In dimensions $3\leq N\leq 9$ the behavior of $U(r)$ is oscillatory around the singular solution $-2\log(r)$. Instead, if $N\geq 10$ then $U(r)<-2\log(r)$ for all $r>0$, since it can be shown that $v_1(s)<\lambda_0$ for all $s\in\R$, as the next result shows.

\begin{claim}\label{7b}
Suppose that $N\geq 10$. Then
\[
0<v_1(s)<\lambda_0,\quad -2<v_2(s)<0,\quad\text{for all } s\in\R.
\]
\end{claim}
\begin{proof}
Associated to the eigenvalue $\mu_+$ in \eqref{5} we have the eigenvector $\xi_1=(\mu_-,1)$. To prove the claim it is enough to show that the curve $(v_1(s),v_2(s)),\ s\in\R$, is between the lines $v_2=-2$ and $v_1=\lambda_0+2\mu_-+\mu_-v_2$ (i.e. the line passing through the point $(\lambda_0,-2)$ in the direction $\xi_1$). 

First suppose, by contradiction, that there exists $s_0\in\R$ such that $0<v_1(s_0)<\lambda_0$ and $v_2(s_0)=-2$. Choose $s_0$ as the smallest one with this property. Recalling \eqref{4a}--\eqref{4bz} and using the minimality condition of $s_0$, we see that 
\[
v'_1(s_0)=0\quad\text{and}\quad v_2'(s_0)\leq 0.
\] 
But, from \eqref{4bz}, we have that $v_2'(s_0)=-v_1(s_0)+2(N-2)>0$, which contradicts the previous. So, if $0<v_1(s)<\lambda_0$ then $v_2(s)>-2$.

On the other hand, suppose, by contradiction, that the curve $(v_1(s),v_2(s))$ crosses the line previously described. So that, there exists $s_0\in\R$ with
\begin{gather}
v_1(s_0)=\mu_-v_2(s_0)+\lambda_0+2\mu_-, \label{7c}\\
\dfrac{1}{\mu_-}\leq\dfrac{v_2'(s_0)}{v_1'(s_0)}.
\end{gather}
 Of course here we are choosing the smallest point where the curves cross each other. Last inequality and \eqref{4bz} yield
 \[
 \dfrac{1}{\mu_-}v_1(v_2+2)\leq-v_1-(N-2)v_2.
 \]
Replacing $v_1$ of \eqref{7c} in the last inequality, we have
 \[
 \mu_-v_2^2+[\mu_-^2+(N-2)\mu_-+\lambda_0+4\mu_-]v_2+2[\mu_-^2+(N-2)\mu_-+\lambda_0]+4\mu_-\geq 0.
 \]
Recalling that $\mu_-$ satisfies 
\[
\mu_-^2+(N-2)\mu_-+\lambda_0=0,
\] 
we deduce that 
\[
v_2^2+4v_2+4\leq 0.
\] 
This contradicts the previous claim. The proof is complete.
\end{proof}

\subsection{Asymptotic behavior} Regarding the function $U$ in \eqref{6}--\eqref{7}, we'll need a further analysis of its asymptotic behavior. In this respect, we have the following result.

\begin{claim}\label{7d}
Let $U$ be the only radial solution to \eqref{6}--\eqref{7} and $v$ be as in \eqref{3c}. Then:
\begin{enumerate}[i)]
\item if $3\leq N\leq 9$, $v(s)=-2s+O(\e^{-\frac{N-2}{2}s})$ as $s\to+\infty$; 
\item if $N=10$, there exist $a\in\R$ and $b<0$ such that 
\[
v(s)=-2s+a\e^{-4s}+bs\e^{-4s}+o(s\e^{-4s})\quad\text{as}\ s\to+\infty;
\]
\item if $N>10$, there exist $a\in\R$ and $b<0$ such that 
\[
v(s)=-2s+a\e^{\mu_-s}+bs\e^{\mu_+s}+o(s\e^{\mu_+s})\quad\text{as}\ s\to+\infty.
\]
\end{enumerate}
\end{claim} 
We shall prove the case of resonance ii), the other cases can be handled similarly.

A preliminary analysis of the autonomous system \eqref{4bz} suggests to consider 
\[
w=v+2s,\quad s\in\R,
\] 
which satisfies 
\begin{equation}\label{7e}
w''+8w'+16w=-16(\e^w-1-w)\quad\text{in}\ \R.
\end{equation}
This implies that there exist constants $a,b\in\R$ such that 
\begin{equation}\label{7f}
w=a\e^{-4s}+bs\e^{-4s}+w_p,
\end{equation}
where $w_p$ is a particular solution of the non-homogeneous equation \eqref{7e} and $a,b$ are numbers depending on $w_p$. 

Following the variation of parameters method, we look for a solution of \eqref{7e} of the form $w_p=u_1\e^{-4s}+u_2s\e^{-4s}$, where
\begin{equation*}
u'_1=\frac{\begin{vmatrix}
0&s\e^{-4s}\\
-16(\e^w-1-w) & \e^{-4s}-4s\e^{-4s}
\end{vmatrix}}{\begin{vmatrix}
\e^{-4s} & s\e^{-4s}\\
-4\e^{-4s} & \e^{-4s}-4s\e^{-4s}
\end{vmatrix}}=16s\e^{4s}(\e^w-1-w)
\end{equation*}
and
\begin{equation*}
u'_2=\frac{\begin{vmatrix}
\e^{-4s}&0\\
-4\e^{-4s}&-16(\e^w-1-w)
\end{vmatrix}}{\begin{vmatrix}
\e^{-4s}&s\e^{-4s}\\
-4\e^{-4s}&\e^{-4s}-4s\e^{-4s}
\end{vmatrix}}=-16\e^{4s}(\e^w-1-w).
\end{equation*}
Considering the expected asymptotic behavior, we choose the particular solution
\begin{equation*}
w_p=-\int_s^{+\infty}16t\e^{4t}(\e^w-1-w)\,\mathrm{d}t\,\e^{-4s}+\int_s^{+\infty}16\e^{4t}(\e^w-1-w)\,\mathrm{d}t\,s\e^{-4s}
\end{equation*}

Claim~\ref{7b} and the asymptotic behavior, as $s\to-\infty$, of $w$ in \eqref{7f} implies that $b\leq 0$. Moreover, $b\neq 0$. Indeed, arguing by contradiction, suppose $b=0$, i.e., 
\[
w=a\e^{-4s}+w_p.
\] 
Multiplying by $e^{4s}$ and differentiating both sides we have
\begin{equation}\label{8a}
(w\e^{4s})'=(w_p\e^{4s})'.
\end{equation}
We can compute right hand side directly from the expression of $w_p$ to obtain that, for all $s\leq 0$,
\[(w_p\e^{4s})'=\int_s^{+\infty}16\e^{4t}(\e^w-1-w)\,\mathrm{d}t\geq\int_0^{+\infty}16\e^{4t}(\e^w-1-w)\,\mathrm{d}t>0.
\]
 On the other hand, 
 \[(w\e^{4s})'=w'\e^{4s}+4w\e^{4s}\to 0\quad\text{as}\ s\to-\infty
 \] 
 (this limit is zero due to the continuity of $U$ and its derivative around the origin). This contradiction implies $b\neq 0$, the proof of the claim is complete.
 

\section{The Method}

The construction that we describe next is motivated by \cite{DPM2007} and \cite{DPMW2008}. Let us consider the change of variables
\[
\widetilde{u}=u\left(\sqrt{\dfrac{\lambda_0}{\lambda}}(x-\xi)\right),
\]
which transform \eqref{1}--\eqref{2} into the equivalent problem
\begin{equation}\label{9}
\left\{\begin{aligned}
-\Delta\widetilde{u}&=\lambda_0\e^{\widetilde{u}} &&\text{in}\ \R^N\setminus\overline{\D}_{\lambda,\xi},\\
\widetilde{u}&=0&&\text{on}\ \partial \D_{\lambda,\xi},
\end{aligned}\right.
\end{equation}
where $\D_{\lambda,\xi}$ is the shrinking domain
\[
\D_{\lambda,\xi}=\xi+\sqrt{\dfrac{\lambda}{\lambda_0}}\D.
\]
The closer $\lambda$ is taken from zero, the ``closer" $\R^N\setminus\overline{\D}_{\lambda,\xi}$ is to $\R^N$, so it is natural to seek for a solution $\widetilde{u}$ in the form of a small perturbation of $U_\alpha$ in \eqref{8}. We need a correction so that the boundary condition is satisfied.  

Let $\varphi_\lambda$ be the solution of the problem
\begin{equation}
 \left\{
  \begin{aligned}
   \Delta\varphi_\lambda&=0 && \text{in } \mathbb{R}^N\setminus\overline{\D}_{\lambda,\xi},\\ \varphi_\lambda(x)&=U_\alpha(x) && \text{on } \partial\D_{\lambda,\xi},\\ \lim_{|x|\to+\infty}&\varphi_\lambda(x)=0. && 
  \end{aligned}
 \right.
\end{equation}
In the same way, consider $\varphi_0$ the function such that 
\begin{equation}\label{9a}
 \left\{
  \begin{aligned}
   \Delta\varphi_0&=0 && \text{in } \mathbb{R}^N\setminus\overline{\D},\\
   \varphi(x)&=1 && \text{on } \partial\D,\\
   \lim_{|x|\to+\infty}&\varphi_0(x)=0.
   \end{aligned}
  \right.
\end{equation}
By the maximum principle, 
 \[
 \varphi_\lambda(x)=(U_\alpha(\xi)+O(\sqrt{\lambda}))\varphi_0\left( \sqrt{\dfrac{\lambda_0}{\lambda}}(x-\xi)\right).
 \]
We also note that
\begin{equation}\label{9b}
f_0:=\lim_{|x|\to+\infty}|x|^{N-2}\varphi_0(x)=\dfrac{1}{(N-2)|S^{N-1}|}\int\limits_{\R^N\setminus\D}|\nabla\varphi_0|^2\,\mathrm{d}x,
\end{equation}
which in particular implies
\begin{equation}\label{9c}
|\varphi_\lambda(x)|\leq C\lambda^{(N-2)/2}|x-\xi|^{2-N}\quad\text{for all}\ x\in\R^N\setminus\D_{\lambda,\xi}.
\end{equation}
The number $\int_{\R^N\setminus\D}|\nabla\varphi_0|^2\,\mathrm{d}x$ is the {\em Newtonian capacity} of $\D$.

Thus we look for a solution to problem \eqref{9} of the form 
\[
\widetilde{u}=U_\alpha-\varphi_\lambda+\phi,
\]
which yields the following equation for $\phi$
\begin{equation}\label{10}
\left\{\begin{aligned}
\Delta\phi+\lambda_0\e^{U_\alpha}\phi&=M(\phi)+E_\lambda&&\text{in}\ \R^N \setminus\overline{\D}_{\lambda,\xi},\\
\phi&=0&&\text{on}\ \partial\D_{\lambda,\xi},
\end{aligned}\right.
\end{equation}
where
\begin{equation}\label{10a}
E_\lambda =\lambda_0\e^{U_\alpha}\varphi_\lambda,\quad M(\phi)=-\lambda_0\e^{U_\alpha}(\e^{\phi-\varphi_\lambda}-1-\phi+\varphi_\lambda).
\end{equation}

\begin{remark}
We emphasize that, until here, this scheme applies to all dimensions $N\geq 3$. However, if $N=3$, in order to solve \eqref{10} we need to choose a special $\xi$, depending on $\lambda$. This will be done in Section~7.  
\end{remark}

For the time being let us consider $N\geq 4$. In order to solve \eqref{10} it is first necessary to construct a bounded right inverse of the linearization of \eqref{1} around $U_\alpha$ in the whole of $\R^N$, this construction is carried out in Section~4. The method has previously been used by Mazzeo and Pacard (see \cite{MP1996}) where solutions with prescribed singular set for subcritical problems are constructed (see \cite{Pac1992} and \cite{Pac1993}).  For small $\lambda>0$ a similar solvability property is established for the  linearized operator around $U_\alpha$ in the exterior domain $\R^N \setminus\overline{\D}_{\lambda,\xi}$. In Section~5 we construct such a bounded right inverse, namely a solution for the linear problem
\begin{equation}\label{11}
\left\{\begin{aligned}
\Delta\phi+\lambda_0\e^{U_\alpha}\phi&=h&&\text{in}\ \R^N \setminus\overline{\D}_{\lambda,\xi},\\
\phi&=0&&\text{on}\ \partial\D_{\lambda,\xi},
\end{aligned}\right.
\end{equation}
for norms on functions $\phi$ and $h$ defined on $\R^N \setminus\overline{\D}_{\lambda,\xi}$ given as follows. For given $0<\sigma<2$ and
\begin{equation}\label{11d}
0<\beta<
\left\{\begin{aligned}
& 1,&& 4\leq N\leq 9,\\
&\min\{\mu_0^-,1\}&&N\geq 10,
\end{aligned}\right.
\end{equation}
where \[\mu_0^-=\dfrac{N-2}{2}-\dfrac{1}{2}\sqrt{(N-2)(N-10)},\] we consider the norms

\begin{gather}
\|\phi\|_{*,\xi}=\|\phi\|_{L^\infty(B_1(\xi))}+\sup_{|x-\xi|\geq 1}|x-\xi|^{\beta}|\phi(x)|,\label{11a}\\
\|h\|_{\ast\ast,\xi}=\sup_{|x-\xi|\leq 1}|x-\xi|^\sigma|h(x)|+\sup_{|x-\xi|\geq 1}|x-\xi|^{2+\beta}|h(x)|.\label{11b}
\end{gather}

In this context there is continuity, as the next result states. 

\begin{proposition}\label{11c}
Assume $N\geq 4$. Then given numbers $\alpha>0$ and $Z>0$, there exist positive constants $C$, $\Lambda$ such that for any $|\xi|\leq Z$ and any $0<\lambda<\Lambda$ the following holds:\\
For any $h$ with $\|h\|_{**,\xi}<+\infty$, there exists a solution of problem \eqref{11} 
\[
\phi=\Psi_\lambda(h),
\] 
which defines a linear operator of $h$ such that 
\[
\|\phi\|_{*,\xi}\leq C\|h\|_{**,\xi}.
\]
\end{proposition}
In Section~6 we use this result and the contraction mapping principle to solve \eqref{10}.


\section{The Operator $\Delta+\lambda_0\e^{U_\alpha}$ in $\R^N$}
Let $U_\alpha$ be a radial solution of \eqref{6}. In this section we study the linear equation 
\begin{equation}\label{12}
 \Delta\phi+\lambda_0\e^{U_\alpha}\phi=h,\quad \text{in}\ \R^N.
\end{equation}
The main result concerns with solvability of this equation and estimates for the solution in the weighted $L^\infty$ norms given by \eqref{11a} and \eqref{11b}. The main result in this section is the following.

\begin{proposition}\label{12b}
Assume $N\geq 4$. Then given $\alpha>0$ and $Z>0$, there exists $C>0$ such that for any $|\xi|\leq Z$ the following holds: For any $h$ with $\|h\|_{**,\xi}<+\infty$, there exists a solution of \eqref{12} \[\phi=\Psi(h),\] which defines a linear operator of $h$ such that
\begin{equation}\label{12c}
\|\phi\|_{\ast,\xi}\leq C\|h\|_{**,\xi}.
\end{equation} 
\end{proposition}
To prove this result we first consider $\xi=0$. We denote the corresponding norms by $\|\|_*$ and $\|\|_{**}$.

\subsection{A Right Inverse}
In this subsection we consider $N\geq 4$ as well as $N=3$, pointing out the main differences between both cases.  

The linear operator in \eqref{12} is of regular singular type and it is well known that it is Fredholm on weighted spaces provided the weight does not equal one of indicial roots (see for instance \cite{Maz1991,MP1996,MS1991}). We include the main points of the argument and omit some technical computations. 

Let us write $h$ as 
\begin{equation*}
h(x)=\sum_{k=0}^\infty h_k(r)\Theta_k(\theta),\ r>0,\ \theta\in S^{N-1},
\end{equation*}
where $\Theta_k$, $k\geq 0$ are the eigenfunctions of the Laplace-Beltrami operator $-\Delta_{S^{N-1}}$ on the sphere $S^{N-1}$, normalized so that they constitute an orthonormal system in $L^2(S^{N-1})$. We take $\Theta_0$ to be a positive constant, associated to the eigenvalue 0 and $\Theta_i,\ 1\leq i\leq N$ is an appropriate multiple of $x_i/|x|$ which has eigenvalue $\lambda_i=N-1,\ 1\leq i\leq N$. In general, $\lambda_k$ denotes the eigenvalue associated to $\Theta_k$, we repeat eigenvalues according to their multiplicity and we arrange them in a non-decreasing sequence. We recall that the set of eigenvalues is given by $\{i(N-2+i)\}_{i\geq 0}$.

We look for a solution $\phi$ to \eqref{12} of the form
\[
\phi(x)=\sum_{k=0}^\infty \phi_k(r)\Theta_k(\theta),\ x=r\theta.
\]
Therefore, $\phi$ satisfies \eqref{12} if and only if 
\begin{equation}\label{12a}
\phi_k''+\frac{N-1}{r}\phi_k'+\left(2(N-2)\e^{U_\alpha}-\frac{\lambda_k}{r^2}\right)\phi_k=h_k, 
\end{equation}
for all $r>0$, for all $k\geq 0$.

To construct solutions of this ODE we need to consider two linearly independent solutions $z_{1,k}$, $z_{2,k}$ of the homogeneous equation 
\begin{equation}\label{13}
\phi_k''+\frac{N-1}{r}\phi_k'+\left(2(N-2)\e^{U_\alpha}-\frac{\lambda_k}{r^2}\right)\phi_k=0,\hspace{0.3cm} r>0.
\end{equation}
Once these generators are identified, the general solution of the equation can be written through the variation of parameters formula as 
\[
\phi_k(r)=z_{1,k}(r)\int z_{2,k}h_kr^{N-1}\,\mathrm{d}r-z_{2,k}(r)\int z_{1,k}h_kr^{N-1}\,\mathrm{d}r, 
\]
where the symbol $\int$ designates arbitrary antiderivatives, which will be specify later.

It is helpful to recall the reduction of order method: If one solution $z_{1,k}$ to \eqref{13} is known, a second linearly independent solution can be found in any interval where $z_{1,k}$ does not vanish as
\begin{equation}\label{13a}
z_{2,k}(r)=z_{1,k}(r)\int z_{1,k}(r)^{-2}r^{1-N}\,\mathrm{d}r.
\end{equation} 

One can find the asymptotic behavior of any solution $z$ of \eqref{13} as $r\to 0$ and as $r\to+\infty$ by examining the indicial roots of the associated Euler equations. We recall \eqref{7a} to get, as $r\to+\infty$, the limiting equation of \eqref{13} 
\begin{equation}\label{14}
r^2\phi_k''+(N-1)r\phi_k'+(2(N-2)-\lambda_k)\phi_k=0, \hspace{0.3cm} k\geq 0.
\end{equation}
As $r\to 0$ the limiting equation is given by 
\[
r^2\phi_k''+(N-1)r\phi_k'-\lambda_k\phi_k=0.
\]
In this way, the behavior will be ruled by $z(r)\sim r^{-\mu}$, where $\mu$ satisfies
\begin{equation}\label{14a}
\mu^2-(N-2)\mu-\lambda_k=0.
\end{equation}

Equation \eqref{12a} can be solved for each $k$ separately:

{\it Case $k=0$}. Since $\lambda_0=0$, Eq.~\eqref{12a} is the radial form of the linear problem \eqref{12}. As $r\to+\infty$ the limiting equation is
\begin{equation}\label{15}
r^2\phi_0''+(N-1)r\phi_0'+2(N-2)\phi_0=0.
\end{equation} 
The indicial roots of the associated Euler equations are
\begin{equation}\label{15a}
\mu_0^\pm=\dfrac{N-2}{2}\pm\dfrac{1}{2}\sqrt{(N-2)(N-10)}.
\end{equation}
As $r\to 0^+$, the indicial roots are 
\begin{equation}\label{15b}
\mu_1=0\quad \text{and}\quad \mu_2=N-2.
\end{equation}

Since Eq.~\eqref{6} is invariant under the transformation $\alpha \mapsto U(\alpha r)+2\log\alpha$, we see by differentiation in $\alpha$ (recall \eqref{4a}) that the function \[z_{1,0}=v_2(\log r)+2\] satisfies \eqref{13}. By Claim~\ref{7d} in Section~2, the asymptotic behavior of $z_{1,0}$, as $r\to+\infty$, depends on the dimension in the following way:  
\begin{enumerate}[i)]
\item if $4\leq N\leq 9$, then $z_{1,0}=O(r^{-\frac{N-2}{2}})$ as $r\to+\infty$ and $z_{1,0}(r)=O(1)$ as $r\to 0^+$;
\item if $N=10$, there exists $c>0$ such that $z_{1,0}=cr^{-4}\log r(1+o(1))$ as $r\to+\infty$ and $z_{1,0}(r)=O(1)$ as $r\to 0^+$; 
\item if $N>10$, there exists $c>0$ such that $z_{1,0}=cr^{-\mu_0^-}(1+o(1))$ as $r\to+\infty$ and $z_{1,0}(r)=O(1)$ as $r\to 0^+$.
\end{enumerate}
Let's construct a second solution to \eqref{13} for each dimension separately. If $4\leq N\leq 9$, define $z_{2,0}$ for small $r>0$ by
\begin{equation}\label{16}
z_{2,0}(r)=z_{1,0}(r)\int_{r_0}^r z_{1,0}^{-2}s^{1-N}\,\mathrm{d}s,
\end{equation}
where $r_0$ is small so that $z_{1,0}>0$ in $(0,r_0)$ (which is possible because $z_{1,0}\sim 1$ near to $0$). Then $z_{2,0}$ is extended to $(0,+\infty)$ so that it is a solution to the homogeneous equation \eqref{13} in this interval. By \eqref{15a} and \eqref{15b}, $z_{2,0}=O(r^{-\frac{N-2}{2}})$ as $r\to+\infty$ and $z_{2,0}\sim r^{2-N}$ as $r\to 0^+$.
We define
\[
\phi_0(r)=z_{1,0}(r)\int_1^r z_{2,0}h_0s^{N-1}\,\mathrm{d}s-z_{2,0}(r)\int_0^r z_{1,0}h_0s^{N-1}\,\mathrm{d}s.
\]
$\phi_0$ depends linearly on $h_0$ and is a solution of \eqref{12a}. We omit a calculation to verify that \[\|\phi_0\|_*\leq C_0\|h_0\|_{**}.\]

If $N\geq 10$, the strategy is the same as previously, but this time is more convenient to rewrite the variation of parameters formula in the form
\[
\phi_0(r)=-z_{1,0}(r)\int_{0}^rz_{1,0}(s)^{-2}s^{1-N}\int_0^sz_{1,0}(\tau)h_0(\tau)\tau^{N-1}\,\mathrm{d}\tau\mathrm{d}s,\ r>0,
\]
This formula is well defined because $z_{1,0}>0$ (see Claim~\ref{7b} in Section~2). Again, a straightforward calculation shows that $\phi_0$ satisfies 
\[
\|\phi_0\|_*\leq C_0\|h_0\|_{**}.
\]

\textit{Case} $k=1,\dotsc,N$. In this case as $r\to+\infty$ eq.~\eqref{13} becomes
\begin{equation}\label{17}
r^2\phi_k''+(N-1)r\phi_k'+(N-3)\phi_k=0.
\end{equation}
The indicial roots of the associated Euler equations are
\begin{equation}\label{17a}
\mu_k^+=N-3\quad \text{and}\quad \mu_k^-=1.
\end{equation}
As $r\to 0^+$, the indicial roots are 
\begin{equation}
\mu_1=-1\quad \text{and}\quad \mu_2=N-1.
\end{equation}
Similarly to the case $k=0$ we have a solution to \eqref{13}, namely $z_{1,k}(r)=-U'_\alpha(r)$ which is positive in all $(0,+\infty)$. Using \eqref{7a}  we find that
\[z_{1,k}=-r^{-1}v_2(\log r).\]
About the behavior of $z_{1,k}$, by \eqref{7a} we deduce that there exist constants $c_\infty, c_0>0$ such that $z_{1,k}=c_\infty r^{-1}(1+o(1))$ as $r\to+\infty$ and $z_{1,k}(r)=c_0r(1+o(r))$ as $r\to 0^+$.
With it we can build a solution to \eqref{12a}
\begin{equation}\label{17b}
\phi_k(r)=-z_{1,k}(r)\int_0^rz_{1,k}(s)^{-2}s^{1-N}\int_0^sz_{1,k}(\tau)h_k(\tau)\tau^{N-1}\,\mathrm{d}\tau\mathrm{d}s.
\end{equation} 
We omit a calculation to show that $\phi_k$ satisfies 
\[
\|\phi_k\|_*\leq C_k\|h_k\|_{**}.
\]

\textit{Case} $k>N$. Define
\begin{equation}\label{17c}
L_k\phi=\phi''+\dfrac{N-1}{r}\phi'+\left(2(N-2)\e^{U_\alpha}-\dfrac{\lambda_k}{r^2}\right)\phi=0.
\end{equation}
This operator satisfies the maximum principle in any interval of the form $(\delta,1/\delta)$, $\delta>0$. Indeed, the positive function $z=-U_\alpha'$ is a supersolution, because 
\[L_kz=\dfrac{N-1-\lambda_k}{r^2}z<0\quad\text{in}\ (0,+\infty),\]
since $\{\lambda_k\}_k$ is an increasing sequence. To prove the solvability of \eqref{12a} in the appropriate space we observe that 
\[
\rho(r)=\pm\dfrac{C_k\|h_k\|_{**}}{r^{\sigma-2}+r^\beta},
\] 
(for some suitable large $C_k$) provides sub and supersolutions to $L_k\phi=h_k$. Then the method of sub and supersolutions shows that $\phi_k$, founded in this way, satisfies \[\|\phi_k\|_*\leq C_k\|h_k\|_{**}.\] 

\begin{remark}[Case $N=3$]\label{17d}~
\begin{enumerate}[i)]
\item Fourier mode $k=0$: is handled exactly as in dimensions $4\leq N\leq 9$. 
\item Fourier modes $k=1,2,3$: due to \eqref{17a} some functions in a subspace of solutions to the homogeneous equation \eqref{13} don't have decay at infinity, as we require. So, in order to solve the non-homogeneous equation \eqref{12a}, we have to impose an orthogonality condition on $h_k$, $k=1,2,3$. If we look at \eqref{17b}, we find out that such an orthogonality condition is
\begin{equation}\label{17e}
\int_0^\infty z_{1,k}(\tau)h_k(\tau)\tau^2\,\mathrm{d}\tau=0,\quad k=1,2,3.
\end{equation}
If so, it follows easily from \eqref{17b} that $\phi_k$ satisfies 
\[
\|\phi_k\|_*\leq C_k\|h_k\|_{**}.
\]
\item Fourier modes $k>3$: the method previously used for higher dimensions works also. 
\end{enumerate}
\end{remark}


\subsection{Continuity}
The previous construction implies that given an integer $m>0$, if $\|h\|_{**}<+\infty$ and $h_k=0$, for all $k\geq m$ then there exists a solution $\phi$ to \eqref{12} that depends linearly with respect to $h$ and 
\[
\|\phi\|_*\leq C_m\|h\|_{**},
\]
 where $C_m$ may depend only in $m$. We shall show that $C_m$ can be chosen independently of $m$ using a blow-up argument that has been previously used by \cite{CHS1984,DPM2007,DPMW2007,DPMW2008,MP1996}.

Suppose, by contradiction, that there is a sequence of functions $h_j$ such that $\|h_j\|_{**}<+\infty$, each $h_j$ has only finitely many non-trivial Fourier modes and that the solution $\phi_j\neq 0$ satisfies 
\[
\|\phi_j\|_*\geq C_j\|h_j\|_{**},
\]
where $C_j\to+\infty$ as $j\to\infty$ (no confusion should arise between $\phi_j$, $h_j$ and the associated Fourier modes). Replacing $\phi_j$ by $\phi_j/\|\phi_j\|_*$ we may assume that $\|\phi_j\|_*=1$ and $\|h_j\|_{**}\to 0$ as $j\to\infty$. We may also assume that the Fourier modes associated to $\lambda_0=0$ and $\lambda_1=...=\lambda_N=N-1$ are zero.

Along a subsequence (which we write the same) we must have 
\begin{equation}\label{18}
\sup_{x>1}|x|^\beta|\phi_j(x)|\geq \dfrac{1}{2}
\end{equation}
or
\begin{equation}\label{19}
\|\phi_j(x)\|_{L^\infty(B_1(0))}\geq \dfrac{1}{2}.
\end{equation}
Assume first that \eqref{18} occurs and let $x_j\in\R^N$ with $|x_j|>1$ be such that 
\[
|x_j|^\beta|\phi_j(x_j)|>\dfrac{1}{4}.
\]
Along a new sequence (denote by the same) $x_j\to x_0$ or $x_j\to+\infty$.

If $x_j\to x_0$ then $x_0\geq 1$ and by standard elliptic estimates $\phi_j\to\phi$ uniformly on compacts sets of $\R^N$. Thus $\phi$ is a solution to \eqref{12} with right hand side equal to zero that also satisfies $\|\phi\|_*<+\infty$ and is such that the Fourier modes $\phi_0,...,\phi_N$ are zero. But the unique solution to this problem is $\phi=0$, contradicting $\phi(x_0)\neq 0$.

If $|x_j|\to+\infty$, consider $\widetilde{\phi}_j(y)=|x_j|^\beta\phi_j(|x_j|y)$. Then $\widetilde{\phi}_j$ satisfies 
\[
\Delta\widetilde{\phi}_j+\lambda_0\e^{U_\alpha(|x_j|y)}|x_j|^2\widetilde{\phi}_j=\widetilde{h}_j\quad\text{in}\ \R^N,
\]
 where $\widetilde{h}_j=|x_j|^{\beta+2}h_j(|x_j|y)$. But since $\|\phi_j\|_*=1$ we have
 \begin{equation}\label{20}
 |\widetilde{\phi}_j(y)|\leq|y|^{-\beta},\quad |y|>\dfrac{1}{|x_j|}.
 \end{equation}
 So $\widetilde{\phi}_j$ is uniformly bounded on compact sets of $\R^N\setminus\{0\}$. Similarly, for $|y|>1/|x_j|$
 \[
 |\widetilde{h}_j(y)|\leq\|h_j\|_*|y|^{-\beta-2}
 \]
 and hence $\widetilde{h}_j\to 0$ uniformly on compact sets of $\R^N\setminus\{0\}$ as $j\to\infty$. By elliptic estimates $\widetilde{\phi}_j\to\phi$ uniformly on compact sets of $\R^N\setminus\{0\}$ and $\phi$ solves
 \[
 \Delta\phi+\lambda_0|y|^{-2}\phi=0\quad\text{in}\ \R^N\setminus\{0\}.
 \]
From \eqref{20} we deduce the bound
\begin{equation}\label{21}
 |\phi(y)|\leq|y|^{-\beta},\quad |y|>0.
\end{equation}

Expanding $\phi$ as 
\[
\phi(x)=\sum_{k=N+1}^\infty \phi_k(r)\Theta_k(\theta),
\]
where $\phi_k$ denotes the Fourier modes of $\phi$ (recall that we assumed at the beginning that the first $N+1$ of these modes were zero), we see that $\phi_k$ has to be a solution to
 \[
 \phi_k''+\frac{N-1}{r}\phi_k'+\dfrac{2(N-2)-\lambda_k}{r^2}\phi_k=0,\quad \forall r>0,\ \forall k>N+1.
 \]
The solutions of this equation are linear combinations of $r^{-\mu_k^\pm}$, where
\[
\mu_k^\pm=\dfrac{N-2}{2}\pm\dfrac{1}{2}\sqrt{(N-2)(N-10)-4\lambda_k},\quad k>N+1.
\]
It's easy to check that $\mu_k^-<0$ and $\beta<\mu_k^+$. Thus, $\phi_k$ cannot have a bound of the form \eqref{21} unless it is identically zero. This is a contradiction because $\widetilde{\phi}_j(x_j/|x_j|)\geq 1/4$ for all $j$.
 
The analysis of the case \eqref{19} is similar. By density, for any $h$ with $\|h\|_{**}<+\infty$ a solution $\phi$ of \eqref{12} can be constructed and it satisfies $\|\phi\|_*\leq C\|h\|_{**}.$ This proves Proposition~\ref{12b} in the case $\xi=0$.
 
 \subsection{Proof of Proposition~\ref{12b}}
Let $\eta$ be a smooth cut-off function such that 
\begin{gather*}
 \eta(x)=0 \quad \text{for all } |x-\xi|\leq\delta,\\
 \eta(x)=1 \quad \text{for all } |x-\xi|\geq 2\delta,
\end{gather*}
where $\delta>0$ is small. We shall solve 
\begin{gather*}
 \Delta\phi_2+\lambda_0\e^{U_\alpha}(1-\eta)\phi_2=(1-\eta)h \quad \text{in } \R^N,\\
 \lim_{|x|\to+\infty}\phi_2(x)=0.
\end{gather*}
Note that for $\delta>0$ sufficiently small but fixed the operator $\Delta+\lambda_0\e^{U_\alpha}(1-\eta)$ is coercive, hence there exists a solution to this problem and we have the estimates
 \begin{gather}
 |\phi_2(x)|\leq C\|h\|_{**,\xi}\quad \text{for all}\ |x-\xi|\leq 1,\label{22}\\
 |\phi_2(x)|\leq C\|h\|_{**,\xi}(1+|x|)^{2-N}\quad \text{for all}\ |x-\xi|\geq 1.\label{23}
 \end{gather}
 
 According to the above arguments, we can solve the equation 
 \begin{equation}\label{24}
 \Delta\phi_1+\lambda_0\e^{U_\alpha}\phi_1=-\lambda_0\e^{U_\alpha}\eta\phi_2+\eta h\quad\text{in}\ \R^N,
 \end{equation}
 provided the right hand side has finite $\|\ \|_{**}$ norm. But, since $\eta\phi_2=0$ for $|x-\xi|\leq\delta$, \eqref{22} and \eqref{23} imply that 
 \[
 \|\lambda_0e^{U_\alpha}\eta\phi_2\|_{**}\leq C\|h\|_{**,\xi}.
 \]
 Thus, there exists a solution $\phi_1$ to \eqref{24}, such that  
 \begin{equation}\label{25}
 \|\phi_1\|_*\leq C\|h\|_{**,\xi}.
 \end{equation}
 Note that the norms $\|\|_{*}$ and $\|\|_{*,\xi}$ are equivalent, as directly can be checked from their definitions. Then there exists $C>0$ (which might depends on $Z$) such that
 \begin{equation}\label{26}
 \|\phi_1\|_{*,\xi}\leq C\|h\|_{**,\xi}.
 \end{equation}
 Define $\phi=\phi_1+\phi_2$, which is a solution to \eqref{12}. Then from \eqref{22}, \eqref{23} and \eqref{26} we see that \eqref{12c} holds, and the proof is complete.\qed


\section{Proof of Proposition~\ref{11c}}
We shall use the operator constructed in the previous section in order to prove Proposition~\ref{11c}. We fix $Z>0$ large and work with $|\xi|\leq Z$. The estimates depend on $\xi$ only through $Z$. We assume that $0\in\D$. Let $0<R_0<R_1$ be fixed such that $2R_0<R_1$ and $\D\subset B_{R_0}$. Let $\rho\in C^\infty(\R^N),\ 0\leq\rho\leq 1$ be such that 
\[\rho(x)=0\quad \text{for}\ |x|\leq 1,\quad\rho(x)=1\quad \text{for}\ |x|\geq 2\]
and set 
\[\eta_\lambda(x)=\rho\left(\dfrac{\lambda_0^{1/2}}{R_0\lambda^{1/2}}(x-\xi)\right),\quad \zeta_\lambda(x)=\rho\left(\dfrac{\lambda_0^{1/2}}{R_1\lambda^{1/2}}(x-\xi)\right).\]

We look for a solution to \eqref{11} of the form 
\[
\phi=\eta_\lambda\varphi+\psi.
\]
We need then to solve the system of equations
\begin{equation}\label{27}
 \left\{
  \begin{aligned}
   \Delta \psi+(1&-\zeta_\lambda)\lambda_0\e^{U_\alpha}\psi\\
   &=-2\nabla\eta_\lambda\nabla\varphi-\varphi\Delta\eta_\lambda+(1-\zeta_\lambda)h &&\text{in } \R^N\setminus\overline{\D}_{\lambda,\xi},\\
   \psi&=0 \quad \text{on } \partial\D_{\lambda,\xi},\quad \lim_{|x|\to+\infty}\psi(x)=0;
  \end{aligned}
 \right.
\end{equation}
\begin{equation}\label{28}
\Delta\varphi+\lambda_0\e^{U_\alpha}\varphi=-\lambda_0\e^{U_\alpha}\zeta_\lambda\psi+\zeta_\lambda h,\quad \text{in } \R^N;
\end{equation}
where $\varphi$, $\psi$ are the unknowns.

Proposition~\ref{11c} will be proved using a fixed point argument. We assume $\|h\|_{**,\xi}<+\infty$. Let \[E_\lambda=B_{2\sqrt{\frac{\lambda}{\lambda_0}} R_0}(\xi)\setminus B_{\sqrt{\frac{\lambda}{\lambda_0}} R_0}(\xi)\] and consider the Banach space \[X=\{\varphi/\ \varphi:\ \R^N \longrightarrow\R\ \text{is Lipschitz continuous in}\ E_\lambda\ \text{with}\ \|\varphi\|_{*,\xi}<+\infty\}\]
with the norm 
\[
\|\varphi\|_X=\|\varphi\|_{*,\xi}+\lambda^{1/2}\|\nabla\varphi\|_{L^\infty(E_\lambda)}.
\]
Given $\varphi\in X$ we first note that \eqref{27} has a solution for suitable small $\lambda$ because $\|(1-\zeta_\lambda)\lambda_0\e^{U_\alpha}\|_{L^{N/2}(\R^N\setminus\overline{\D}_{\lambda,\xi})}\to 0$ as $\lambda\to 0^+$. Let $\psi(\varphi)$ denote this solution, which is clearly linear in $\varphi$. As we shall see, $|\psi|\leq C/|x|^{N-2}$ for large $|x|$, which implies that the right hand side of \eqref{28} has a finite $\|\ \|_{**,\xi}$. Then, by Proposition~\ref{12b}, Eq.~\eqref{28} has a solution $\overline{\varphi}$ such that $\|\overline{\varphi}\|_{*,\xi}<+\infty$. Set $F(\varphi)=\overline{\varphi}$.

For $\varphi\in X$ we will fist prove the estimate
\begin{equation}\label{28b}
|\psi(x)|\leq C\lambda^{(N-2)/2}(\|h\|_{**,\xi}+\|\varphi\|_X)|x-\xi|^{2-N},
\end{equation}
for all $x\in\R^N\setminus\overline{\D}_{\lambda,\xi}$. Indeed, let $\widetilde{\psi}(z)=\psi\left(\xi+\sqrt{\dfrac{\lambda}{\lambda_0}}z\right),\ z\in\R^N\setminus\D.$ Then
\begin{equation}\label{28a}
\left\{\begin{gathered} 
\Delta \widetilde{\psi}+\lambda(1-\rho(z/{R_1}))\e^{U_\alpha}\widetilde{\psi}=g\quad\text{in}\ \R^N\setminus\overline{\D},\\
\widetilde{\psi}=0\quad\text{on}\ \partial\D,\quad \lim_{|x|\to+\infty}\widetilde{\psi}(x)=0,
\end{gathered}\right.
\end{equation}
where 
\begin{multline*}
g=-2\dfrac{\lambda^{1/2}}{R_0\lambda_0^{1/2}}\nabla\rho\left(\dfrac{z}{R_0}\right)\nabla\varphi\left(\xi+\sqrt{\dfrac{\lambda}{\lambda_0}}z\right)-\dfrac{1}{R_0^2}\Delta\rho\left(\dfrac{z}{R_0}\right)\varphi\left(\xi+\sqrt{\dfrac{\lambda}{\lambda_0}}z\right)\\
+\dfrac{\lambda}{\lambda_0}\left(1-\rho\left(\dfrac{z}{R_1}\right)\right)h\left(\xi+\sqrt{\dfrac{\lambda}{\lambda_0}}z\right).
\end{multline*}
Then the support of $g$ is contained in the ball $B_{2R_1}$ and we can estimate for all $z\in\R^N\setminus\D$, $|z|\leq 2R_1,$
\begin{gather}
2\dfrac{\lambda^{1/2}}{R_0\lambda_0^{1/2}}\left|\nabla\rho\left(\dfrac{z}{R_0}\right)\nabla\varphi\left(\xi+\sqrt{\dfrac{\lambda}{\lambda_0}}z\right)\right|\leq C\|\varphi\|_X \label{29}\\
\dfrac{1}{R_0^2}\left|\Delta\rho\left(\dfrac{z}{R_0}\right)\varphi\left(\xi+\sqrt{\dfrac{\lambda}{\lambda_0}}z\right)\right|\leq C\|\varphi\|_X \label{30}\\
\dfrac{\lambda}{\lambda_0}\left|\left(1-\rho\left(\dfrac{z}{R_1}\right)\right)h\left(\xi+\sqrt{\dfrac{\lambda}{\lambda_0}}z\right)\right|\leq C\lambda^{1-\sigma/2}\|h\|_{**,\xi}. \label{31}
\end{gather}
Since $0\in\D$ and $\sigma<2$, we see from \eqref{29}--\eqref{31} that 
\[
|g(z)|\leq C(\|\varphi\|_X+\|h\|_{**,\xi})\chi_{2R_1}.
\]
This estimate and \eqref{28a} yield 
\[
|\widetilde{\psi}(z)|\leq C(\|\varphi\|_X+\|h\|_{**,\xi})|z|^{2-N}\quad\text{for all}\ z\in\R^N\setminus\D
\]
which implies \eqref{28b}.

Recall that $\varphi\in X$, $\psi=\psi(\varphi)$ is the solution to \eqref{27} and we use the notation $\overline{\varphi}=F(\varphi)$. 

By Proposition~\ref{12b} we have
\begin{equation}\label{31a}
\|\overline{\varphi}\|_{*,\xi}\leq C(\|\lambda_0\e^{U_\alpha}\zeta_\lambda\psi\|_{**,\xi}+\|\zeta_\lambda h\|_{**,\xi}).
\end{equation}
Using \eqref{28b} we can estimate $\|\lambda_0\e^{U_\alpha}\zeta_\lambda\psi\|_{**,\xi}$. We have
\begin{equation}\label{32}
 \begin{aligned}
  \sup&_{|x-\xi|\leq 1}|x-\xi|^\sigma\e^{U_\alpha}\zeta_\lambda|\psi|\\
   & \leq C\lambda^{(N-2)/2}(\|h\|_{**,\xi}+\|\varphi\|_X)\sup_{\sqrt{\lambda/\lambda_0}R_1\leq|x-\xi|\leq 1}|x-\xi|^{2-N+\sigma}\\
 &\leq C\lambda^{\sigma/2}(\|h\|_{**,\xi}+\|\varphi\|_X).
 \end{aligned}
\end{equation}
On the other hand
\begin{equation}\label{33}
 \begin{aligned}
  \sup_{|x-\xi|\geq 1}|x-&\xi|^{2+\beta}\e^{U_\alpha}\zeta_\lambda|\psi|\\
  &\leq C\lambda^{(N-2)/2}(\|h\|_{**,\xi}+\|\varphi\|_X)\sup_{|x-\xi|\geq 1}|x-\xi|^{2-N+\beta}\\
  &\leq C\lambda^{(N-2)/2}(\|h\|_{**,\xi}+\|\varphi\|_X).
 \end{aligned}
\end{equation}

We deduce from \eqref{32} and \eqref{33} that
\begin{equation}\label{34}
\|\lambda_0\e^{U_\alpha}\zeta_\lambda\psi\|_{**,\xi}\leq C\lambda^{\sigma/2}(\|h\|_{**,\xi}+\|\varphi\|_X). 
\end{equation}
Therefore, from \eqref{31a} and \eqref{34}, we find that
\[
\|\overline{\varphi}\|_{*,\xi}\leq C(\lambda^{\sigma/2}\|\varphi\|_X+\|h\|_{**,\xi}).
\]
Using a scaling argument and elliptic estimates we can prove
\[
\sup_{E_\lambda}|\nabla\overline{\varphi}|\leq C\lambda^{-1/2}\|\overline{\varphi}\|_{*,\xi}
\]
and hence 
\[
\|F(\varphi)\|_X\leq C(\lambda^{\sigma/2}\|\varphi\|_X+\|h\|_{**,\xi}).
\]
Since $F$ is affine, this estimate shows that $F$ has a unique fix point $\varphi$ in $X$ for $\lambda>0$ suitable small, and the fix point satisfies 
\[
\|\varphi\|_X\leq C\|h\|_{**,\xi}.
\]
\qed


\section{Proof of Theorem~\ref{4d}}
In this section we prove Theorem~\ref{4d} by using a fixed point argument to solve problem \eqref{10}. In particular, we prove the following result:
\begin{proposition}
Assume $N\geq 4$. Then given $\alpha>0$ and $Z>0$, there are positive numbers $\Lambda$, $C$ such that for any $|\xi|<Z$ and any $0<\lambda<\Lambda$, there exists $\phi_{\lambda,\xi}$ solution to problem \eqref{10} such that 
\begin{equation}\label{35}
\|\phi_{\lambda,\xi}\|_{*,\xi}\leq C\lambda^{\sigma/2}\quad \text{for all}\ 0<\lambda<\Lambda,\ |\xi|<Z.
\end{equation}
\end{proposition}
\begin{proof}
There is not loss of generality in assuming $0\in \D$. Fix $\delta>0$ such that $B_\delta(0)\subset\D$.
We first estimate $\|E_\lambda\|_{**,\xi}$ and $\|M(\phi)\|_{**,\xi}$ in \eqref{10a}. In particular we have
\begin{gather}
\|E_\lambda\|_{**,\xi}\leq C\lambda^{\sigma/2} \label{36}\\ 
\|M(\phi)\|_{**,\xi}\leq C(\|\phi\|_{*,\xi}^2+\lambda^{\sigma/2})\e^{\|\phi\|_{*,\xi}}. \label{37}
\end{gather}

In fact, by \eqref{9c}
\begin{align*}
\sup_{|x-\xi|\leq 1,\ x\notin \D_{\lambda,\xi}}|x-\xi|^\sigma|\varphi_\lambda(x)|\lambda_0\e^{U_\alpha}&\leq C\lambda^{(N-2)/2}\sup_{\delta\sqrt{\lambda/\lambda_0}\leq|x-\xi|\leq 1}|x-\xi|^{\sigma+2-N}\\
&\leq C\lambda^{\sigma/2},
\end{align*}
and
\begin{align*}
\sup_{|x-\xi|\geq 1}|x-\xi|^{2+\beta}|\varphi_\lambda(x)|\lambda_0\e^{U_\alpha}&\leq C\lambda^{(N-2)/2}\sup_{|x-\xi|\geq 1}|x-\xi|^{\beta+2-N}\\
&\leq C\lambda^{(N-2)/2},
\end{align*}
which yields \eqref{36}.

For \eqref{37}, by the definition of $M$ and the identity $\e^\epsilon=1+\epsilon+\int_0^\epsilon\e^t(\epsilon-t)\,\textrm{d}t$, valid for all $\epsilon\in\R$, we have 
\[
|M(\phi)|\leq C\e^{U_\alpha}(\phi^2+\varphi_\lambda^2)\e^{|\phi|+|\varphi_\lambda|}.
\]
Additionally,
\begin{equation*}
\sup_{|x-\xi|\leq 1,\ x\notin \D_{\lambda,\xi}}|x-\xi|^\sigma\e^{U_\alpha}\phi^2\leq C\|\phi\|_{*,\xi}^2
\end{equation*} 
and
\begin{align*}
\sup_{|x-\xi|\leq 1,\ x\notin \D_{\lambda,\xi}}|x-\xi|^\sigma\e^{U_\alpha}\varphi_\lambda^2&\leq C\lambda^{N-2} \sup_{\delta\sqrt{\lambda/\lambda_0}\leq|x-\xi|\leq 1}|x-\xi|^{\sigma+4-2N}\\
&\leq C\lambda^{\sigma/2}.
\end{align*}
Note also that, by \eqref{9c},  
\[
|\varphi_\lambda(x)|\leq C\delta^{2-N}\quad\text{for all}\ x\notin\D_{\lambda,\xi},\ \lambda>0.
\]
These inequalities yield
\begin{equation}\label{38}
\sup_{|x-\xi|\leq 1,\ x\notin \D_{\lambda,\xi}}|x-\xi|^\sigma|M(\phi)|\leq C(\|\phi\|_{*,\xi}^2+\lambda^{\sigma/2})\e^{\|\phi\|_{*,\xi}}
\end{equation}
On the other hand
\begin{align*}
\sup_{|x-\xi|\geq 1}|x-\xi|^{2+\beta}\e^{U_\alpha}\phi^2&\leq C\|\phi\|_{*,\xi}^2\sup_{|x-\xi|\geq 1}|x-\xi|^{-\beta}\\
&\leq C\|\phi\|_{*,\xi}^2
\end{align*}
and
\begin{align*}
\sup_{|x-\xi|\geq 1}|x-\xi|^{2+\beta}\e^{U_\alpha}\varphi_\lambda^2&\leq C\lambda^{N-2}\sup_{|x-\xi|\geq 1}|x-\xi|^{\beta+4-2N}\\
&\leq C\lambda^{N-2}.
\end{align*}
Then 
\begin{equation}\label{39}
\sup_{|x-\xi|\geq 1}|x-\xi|^{2+\beta}|M(\phi)|\leq C(\|\phi\|_{*,\xi}^2+\lambda^{N-2})\e^{\|\phi\|_{*,\xi}}.
\end{equation}
Combining \eqref{38} with \eqref{39} we obtain \eqref{37}.

Now let us focus on the fixed point argument. We define for small $\rho>0$
\[
\mathcal{F}=\{\phi:\ \R^N\setminus\D_{\lambda,\xi} \longrightarrow\R\ /\ \|\phi\|_{*,\xi}\leq\rho\}
\]
and the operator $\overline{\phi}=\mathcal{A}(\phi)$ where $\overline{\phi}$ is the solution of Proposition~\ref{11c} to
\begin{equation}
\left\{\begin{aligned}
\Delta\overline{\phi}+\lambda_0\e^{U_\alpha}\overline{\phi}&=M(\phi)+E_\lambda&&\text{in}\ \R^N \setminus\overline{\D}_{\lambda,\xi},\\
\phi&=0&&\text{on}\ \partial\D_{\lambda,\xi},
\end{aligned}\right.
\end{equation}
where $M$ and $E_\lambda$ are given by \eqref{10a}. We prove that choosing $\rho>0$ small enough, $\mathcal{A}$ has a fixed point in $\mathcal{F}$. From Proposition~\ref{11c} we have the estimate
\[
\|\mathcal{A}(\phi)\|_{*,\xi}\leq C(\|M(\phi)\|_{**,\xi}+\|E_\lambda\|_{**,\xi})
\]
and, by \eqref{36} and \eqref{37},
\[
\|\mathcal{A}\|_{*,\xi}\leq C(\rho^2\e^\rho+\lambda^{\sigma/2}\e^\rho+\lambda^{\sigma/2})\leq\rho,
\]
if $\rho>0$ is fixed suitable small and then one consider $\lambda\to 0^+$. This proves $\mathcal{A}(\mathcal{F})\subset\mathcal{F}$.

Now let us take $\phi_1$ and $\phi_2$ in $\mathcal{F}$. Then
\begin{equation}\label{39a}
\|\mathcal{A}(\phi_1)-\mathcal{A}(\phi_2)\|_{*,\xi}\leq C\|M(\phi_1)-M(\phi_2)\|_{**,\xi}.
\end{equation}
To estimate the right hand side, consider $\overline{\phi}\in(\phi_1,\phi_2)\cup(\phi_2,\phi_1)$ such that  
\[
M(\phi_1)-M(\phi_2)=M'(\overline{\phi})(\phi_1-\phi_2).
\]

Directly from the definition of $M$, we compute 
\[
M'(\phi)=-\lambda_0\e^{U_\alpha}(\e^{\phi-\varphi_\lambda}-1).
\]
Indeed, note that \[|M'(\phi)|\leq C\e^{U_\alpha}(|\phi|+|\varphi_\lambda|)\e^{|\phi|+|\varphi_\lambda|}.\]
Therefore, 
\[
 |M(\phi_1)-M(\phi_2)|\leq C\e^{U_\alpha}(|\overline{\phi}|+|\varphi_\lambda|)\e^{|\overline{\phi}|}|\phi_1-\phi_2|.
\]
Similarly to \eqref{36} and \eqref{37},
\begin{equation*}
\sup_{|x-\xi|\leq 1,\ x\notin \D_{\lambda,\xi}}|x-\xi|^\sigma\e^{U_\alpha}|\overline{\phi}|\e^{|\overline{\phi}|}|\phi_1-\phi_2|\leq C\rho\e^\rho\|\phi_1-\phi_2\|_{*,\xi}
\end{equation*}
and
\begin{align*}
\sup_{|x-\xi|\leq 1,\ x\notin \D_{\lambda,\xi}}|x&-\xi|^\sigma\e^{U_\alpha}|\varphi_\lambda|\e^{|\overline{\phi}|}|\phi_1-\phi_2|\\ &\leq C\e^\rho\lambda^{(N-2)/2}\sup_{\delta\sqrt{\lambda/\lambda_0}\leq|x-\xi|\leq 1}|x-\xi|^{\sigma+2-N}\|\phi_1-\phi_2\|_{*,\xi}\\
&\leq C\e^{\rho}\lambda^{\sigma/2}\|\phi_1-\phi_2\|_{*,\xi}.
\end{align*}
These inequalities yield 
\begin{equation}\label{40}
\sup_{|x-\xi|\leq 1,\ x\notin \D_{\lambda,\xi}}|x-\xi|^\sigma|M(\phi_1)-M(\phi_2)|\leq C(\rho+\lambda^{\sigma/2})\e^\rho\|\phi_1-\phi_2\|_{*,\xi}.
\end{equation}

On the other hand
\begin{equation*}
\sup_{|x-\xi|\geq 1}|x-\xi|^{2+\beta}\e^{U_\alpha}|\overline{\phi}|\e^{|\overline{\phi}|}|\phi_1-\phi_2|\leq C\rho\e^\rho\|\phi_1-\phi_2\|_{*,\xi}
\end{equation*}
and
\begin{align*}
\sup_{|x-\xi|\geq 1}|x&-\xi|^{2+\beta}\e^{U_\alpha}|\varphi_\lambda|\e^{|\overline{\phi}|}|\phi_1-\phi_2|\\
&\leq C\lambda^{(N-2)/2}\e^\rho\sup_{|x-\xi|\geq 1}|x-\xi|^{2-N}\|\phi_1-\phi_2\|_{*,\xi}\\
&=C\lambda^{(N-2)/2}\e^\rho\|\phi_1-\phi_2\|_{*,\xi}.
\end{align*}
Then 
\begin{equation}\label{41}
\sup_{|x-\xi|\geq 1}|x-\xi|^{2+\beta}|M(\phi_1)-M(\phi_2)|\leq C(\rho+\lambda^{(N-2)/2})\e^\rho\|\phi_1-\phi_2\|_{*,\xi}.
\end{equation}
Combining \eqref{40} with \eqref{41} we obtain 
\begin{equation}\label{42}
\|M(\phi_1)-M(\phi_2)\|_{**,\xi}\leq C(\rho+\lambda^{\sigma/2})\e^\rho\|\phi_1-\phi_2\|_{*,\xi}.
\end{equation}

Gathering \eqref{39a} and \eqref{42} we conclude that $\mathcal{A}$ is a contraction mapping in $\mathcal{F}$ provided $\rho>0$ is fixed suitable small, and hence it has unique fixed point in this set. Moreover, from the previous steps we deduce the estimate
\[
\|\phi_{\lambda,\xi}\|_{*,\xi}\leq C\lambda^{\sigma/2}\quad \text{for all } 0<\lambda<\Lambda,
\]
which is the desired conclusion.
\end{proof}


\section{The case $N=3$}
In this section, we show the modifications needed in Theorem~\ref{4d} and its proof for the low dimension case, so we consider without mentioning $N=3$. 

We use again the norms defined in \eqref{11a}--\eqref{11b}, but this time $\beta\in(0,1/2)$. As we pointed out in Remark~\ref{17d}, the problem 
\[
 \Delta \phi-\lambda_0\e^{U_\alpha}\phi=h\quad\text{in}\ \R^3,\quad \|h\|_{**,\xi}<+\infty,
\] 
may not be solvable for $\|\phi\|_{*,\xi}<+\infty$, unless $h$ satisfies the orthogonality conditions
\begin{equation}\label{42b}
\int\limits_{\R^3}h\dfrac{\partial U_\alpha}{\partial x_i}\,\mathrm{d}x=0,\quad i=1,2,3,
\end{equation}
(note that these conditions are equivalent to those in \eqref{17e}).

Therefore, problem \eqref{10} may not be solvable in the required space unless $\xi$ would be chosen in a very special way. So, in low dimension we consider instead the projected problem
\begin{equation}\label{43}
\left\{\begin{aligned}
\Delta\phi+\lambda_0\e^{U_\alpha}\phi&=M(\phi)+E_\lambda+\sum_{i=1}^3c_i\Phi_i&&\text{in}\ \R^3 \setminus\overline{\D}_{\lambda,\xi},\\
\phi&=0&&\text{on } \partial\D_{\lambda,\xi},
\end{aligned}\right.
\end{equation}
where $c_i$'s are constants, which are part of the unknown, and 
\[
\Phi_i(x)=\eta(x)\dfrac{\partial U_\alpha}{\partial x_i}(x),\quad i=1,2,3.
\]
$\eta$ is a fixed radial cut-off function, i.e. $\eta\in C^\infty(\R^3)$, $\eta(x)=\eta(|x|)$, $0\leq\eta\leq 1$ and 
\[
\eta(x)=1\ \text{for}\ |x|\leq 1,\quad \eta(x)=0\ \text{for}\ |x|\geq 2.
\] 
The only purpose of $\eta$ is to make $\Phi_i$ ``sufficiently" integrable in $\R^3$.

We handle problem \eqref{43} using a similar scheme to problem \eqref{10}. Through an application of the Banach fixed point theorem in a suitable $L^\infty$ space, we prove that \eqref{43} is solvable in the form $\phi=\phi(\lambda,\xi)$, $c_i=c_i(\lambda,\xi)$, where the dependence on the parameters is continuous. We then obtain a solution of problem \eqref{10} if \[c_i(\lambda,\xi)=0\quad\text{for all}\ i=1,2,3.\] We will show that for each sufficiently small $\lambda$ there is indeed a point $\xi$ such that this system of equations is satisfied.

Similarly to higher dimensions, the use of contraction mapping principle is based on the construction of a bounded inverse for the linear problem 
\begin{equation}\label{44}
\left\{\begin{aligned}
\Delta\phi+\lambda_0\e^{U_\alpha}\phi&=h+\sum_{i=1}^3c_i\Phi_i&&\text{in}\ \R^3 \setminus\overline{\D}_{\lambda,\xi},\\
\phi&=0&&\text{on}\ \partial\D_{\lambda,\xi},
\end{aligned}\right.
\end{equation}
We have this analogous result to Proposition~\ref{11c}.
\begin{proposition}\label{44a}
Let us consider numbers $\alpha>0$ and $Z>0$. Then there exist positive constants $C$, $\Lambda$ such that for any $|\xi|\leq Z$ and any $0<\lambda<\Lambda$ the following holds: For any $h$ with $\|h\|_{**,\xi}<+\infty$, there exists a solution of problem \eqref{44} 
\[
(\phi,c_1,c_2,c_3)=\Psi_\lambda(h),
\] 
which defines a linear operator of $h$ such that 
\[
\|\phi\|_{*,\xi}+\max_{i=1,2,3}|c_i|\leq C\|h\|_{**,\xi}.
\]
\end{proposition}
As we did in Section~4, we first consider the version of problem \eqref{44} in entire space,
\begin{equation}\label{45}
 \Delta\phi+\lambda_0\e^{U_\alpha}\phi=h+\sum_{i=1}^3c_i\Phi_i\quad \text{in}\ \R^3.
\end{equation}
The corresponding result to Proposition~\ref{12b} is the following.
\begin{proposition}\label{45a}
Let $\alpha>0$ and $Z>0$. Then there exists a $C>0$ such that for any $|\xi|\leq Z$ the following holds: For any $h$ with $\|h\|_{**,\xi}<+\infty$, there exists a solution of \eqref{45}
\[
(\phi,c_1,c_2,c_3)=\Psi(h),
\]
which defines a linear operator of $h$ such that
\begin{equation}\label{46}
\|\phi\|_{\ast,\xi}+\max_{i=1,2,3}|c_i|\leq C\|h\|_{**,\xi}.
\end{equation} 
\end{proposition}

We observe that the numbers $c_i$ are explicit functions of $h$. Indeed, if $\phi$ solves \eqref{45} with the bound \eqref{46} then two integrations by parts again $\Phi_i$ yield 
\begin{equation}\label{47}
c_i=-\frac{\int_{\R^3}\! h\Phi_i\,\mathrm{d}x}{\int_{\R^3}\!\eta\left|\tfrac{\partial U_\alpha}{\partial x_i}\right|^2\mathrm{d}x},\quad i=1,2,3.
\end{equation}
This expression allows us to estimate $|c_i|$ in terms of $\|h\|_{**,\xi}$.

The scheme of the proof of Proposition~\ref{45a} is analogous to the one used in Proposition~\ref{12b}. We first consider $\xi=0$ and write $h$ in its Fourier modes. Then we treat each Fourier mode of Eq.~\eqref{45} separately. For Fourier modes $k=1,2,3$, we have to take care of choosing $c_i$ according to \eqref{47}; in this way, orthogonality conditions \eqref{42b}, and then \eqref{17e}, will be satisfied. The estimates for $|c_i|$, $i=1,2,3,$ in \eqref{46}  are obtained using \eqref{47}. The blow-up method used to prove the continuity of the operator, as well as the gluing argument are similar to Section~4, we omit the details.

Likewise, we can prove Proposition~\ref{44a} from Proposition~\ref{45a} using a similar scheme to Section~5.

\subsection{Proof of Theorem~\ref{42a}}    
Using a similar scheme to Section~6, from Proposition~\ref{44a} we can prove the existence of solutions to problem \eqref{43} in low dimension, we omit the details. In particular, we have
\begin{proposition}
Let us consider $\alpha>0$ and $Z>0$. Then there are positive numbers $\Lambda$, $C$ such that for any $|\xi|<Z$ and any $0<\lambda<\Lambda$ there exist $\phi_{\lambda,\xi},\ c_1(\lambda,\xi),\ c_2(\lambda,\xi),\ c_3(\lambda,\xi)$ solution to problem \eqref{43} such that 
\begin{equation}\label{47a}
\|\phi_{\lambda,\xi}\|_{*,\xi}+\max_{i=1,2,3}|c_i(\lambda,\xi)|\leq C\lambda^\gamma\quad \text{for all}\ 0<\lambda<\Lambda,\ |\xi|<Z,
\end{equation}
where 
\[
\gamma=1/2\min\{\sigma,1\}.
\]
\end{proposition}

Next we make a remark on how to recognize when $c_i=0$ in Eq.~\eqref{44}.
\begin{lemma}
There is $\epsilon_0>0$ small such that if $\lambda<\epsilon_0$ and $\phi$ is a solution to \eqref{44} such that $\|\phi\|_{*,\xi}<+\infty$, $\|h\|_{**,\xi}<+\infty$, then $c_i=0$ for all $i=1,2,3$ if and only if
\begin{equation*}
\int\limits_{\partial\D_{\lambda,\xi}}\!\dfrac{\partial\phi}{\partial n}\dfrac{\partial U_\alpha}{\partial x_i}\,\mathrm{dS}(x)+\int\limits_{\R^3\setminus\D_{\lambda,\xi}}\! h\dfrac{\partial U_\alpha}{\partial x_i}\,\mathrm{d}x=0\quad\text{for all}\ i=1,2,3.
\end{equation*}
\end{lemma}
\begin{proof}
Since $\partial_{x_j}U_\alpha$ satisfies the linear homogeneous equation in $\R^3$, multiplying \eqref{44} by this function and integrating by parts in $B_R(0)\setminus\D_{\lambda,\xi}$, where $R$ is large, we obtain 
\begin{equation}\label{48}
 \begin{aligned}
  \int\limits_{\partial(B_R(0)\setminus\D_{\lambda,\xi})}\!\left(\dfrac{\partial\phi}{\partial n}\dfrac{\partial   U_\alpha}{\partial x_j}\right.&-\left.\phi\dfrac{\partial}{\partial n}\dfrac{\partial U_\alpha}{\partial x_j}\right)\,\mathrm{dS}(x)\\
  &=\int\limits_{B_R(0)\setminus\D_{\lambda,\xi}}\!\left(h+\sum_{i=1}^3c_i\Phi_i\right)\dfrac{\partial U_\alpha}{\partial x_j}\,\mathrm{d}x.
 \end{aligned}
\end{equation}

Since $\|\phi\|_{*,\xi}<+\infty$ we see that 
\[|
\phi(x)|\leq C|x|^{-\beta}\quad\text{for all}\ |x|\geq R'.
\]
A scaling argument and elliptic estimates show that 
\[
|\nabla\phi(x)|\leq C|x|^{-\beta-1}\quad\text{for all}\ |x|\geq R',
\] 
where $R'>0$ is a large fixed number. Thus
\[
\left|\dfrac{\partial\phi}{\partial n}\dfrac{\partial U_\alpha}{\partial x_j}-\phi\dfrac{\partial}{\partial n}\dfrac{\partial U_\alpha}{\partial x_j}\right|\leq C|x|^{-\beta-2}\quad\text{for all}\ |x|\geq R',
\]
and hence
\[
\lim_{R\to+\infty}\int\limits_{\partial B_R(0)}\left(\dfrac{\partial\phi}{\partial n}\dfrac{\partial U_\alpha}{\partial x_j}-\phi\dfrac{\partial}{\partial n}\dfrac{\partial U_\alpha}{\partial x_j}\right)\,\mathrm{dS}(x)=0.
\]

Letting $R\to+\infty$ in \eqref{48} yields
\[
\sum_{i=1}^3c_i\int\limits_{\R^3\setminus\D_{\lambda,\xi}}\Phi_i\dfrac{\partial U_\alpha}{\partial x_j}\,\mathrm{d}x=-\int\limits_{\R^3\setminus\D_{\lambda,\xi}}h\dfrac{\partial U_\alpha}{\partial x_j}\,\mathrm{d}x-\int\limits_{\partial\D_{\lambda,\xi}}\dfrac{\partial\phi}{\partial n}\dfrac{\partial U_\alpha}{\partial x_j}\,\mathrm{dS}(x).
\]
For $\lambda>0$ sufficiently small the matrix with entries $\int_{\R^3\setminus\D_{\lambda,\xi}}\Phi_i\frac{\partial U_\alpha}{\partial x_j}\,\mathrm{d}x$ is close to $\int_{\R^3}\Phi_i\frac{\partial U_\alpha}{\partial x_j}\,\mathrm{d}x$ which is invertible. The lemma follows.
\end{proof}

\textit{Seeking} $c_i=0$. Finally we show how to find a $\xi=\xi(\lambda)$, $\lambda>0$ small, such that 
\[
 c_i(\lambda,\xi)=0\quad\text{for all}\ i=1,2,3,
\] 
and thereby prove Theorem~\ref{42a}. 

Let us assume $0\in\D$ and $\sigma\in(1,2)$. We have found a solution $\phi_{\lambda,\xi}$, $c_1(\lambda,\xi)$,  $c_2(\lambda,\xi)$, $c_3(\lambda,\xi)$ to \eqref{43}. By the previous lemma, for all $\lambda$ small $c_1=c_2=c_3=0$ if and only if 
\begin{equation}
\int\limits_{\R^3\setminus\D_{\lambda,\xi}}\!(E_\lambda+M(\phi_{\lambda,\xi}))\dfrac{\partial U_\alpha}{\partial x_i}\,\mathrm{d}x+\int\limits_{\partial\D_{\lambda,\xi}}\!\dfrac{\partial\phi_{\lambda,\xi}}{\partial n}\dfrac{\partial U_\alpha}{\partial x_i}\,\mathrm{dS}(x)=0
\end{equation}
for all $i=1,2,3$. 

Let us define
\begin{equation}\label{49}
G_j(\xi)=\int\limits_{\R^3\setminus\D_{\lambda,\xi}}\!(E_\lambda+M(\phi_{\lambda,\xi}))\dfrac{\partial U_\alpha}{\partial x_i}\,\mathrm{d}x+\int\limits_{\partial\D_{\lambda,\xi}}\!\dfrac{\partial\phi_{\lambda,\xi}}{\partial n}\dfrac{\partial U_\alpha}{\partial x_i}\,\mathrm{dS}(x).
\end{equation} 
Using local uniqueness, the fixed point characterization of $\phi_\lambda$ and elliptic estimates, one can prove that the functions $G_j$ are continuous; we omit the details. 

Recalling the definition of $f_0$ in \eqref{9b}, we claim that
\begin{equation}\label{49a}
G_j(\xi)=f_0\lambda^{1/2}\int\limits_{\R^3}|x-\xi|^{-1}\e^{U_\alpha}\dfrac{\partial U_\alpha}{\partial x_j}\,\mathrm{d}x+o(\lambda^{1/2})
\end{equation}
uniformly on compacts sets of $\R^3$. Then, for $\lambda$ small $G_j(\xi)\sim f_0\lambda^{1/2}\frac{\partial U_\alpha}{\partial x_j}(\xi)$, so we can expect that there exists $\xi$ annulling the functions $G_j$, $j=1,2,3$. 

We first observe that
\begin{equation}\label{49b}
\int\limits_{\R^3\setminus\D_{\lambda,\xi}}M(\phi_{\lambda,\xi})\dfrac{\partial U_\alpha}{\partial x_i}\,\mathrm{d}x=O(\lambda^{\sigma/2}).
\end{equation}
Indeed,
\begin{equation*}
\int\limits_{\R^3\setminus\D_{\lambda,\xi}}\!\left|M(\phi_{\lambda,\xi})\dfrac{\partial U_\alpha}{\partial x_i}\right|\,\mathrm{d}x=\int\limits_{B_1(\xi)\setminus\D_{\lambda,\xi}}\!\cdots\,\mathrm{d}x+\int\limits_{\R^3\setminus B_1(\xi)}\!\cdots\,\mathrm{d}x;
\end{equation*}
by \eqref{38} and \eqref{47a},
\begin{align*}
\int\limits_{B_1(\xi)\setminus\D_{\lambda,\xi}}\left|M(\phi_{\lambda,\xi})\dfrac{\partial U_\alpha}{\partial x_i}\right| \,\mathrm{d}x&\leq C(\|\phi\|_{*,\xi}^2+\lambda^{\sigma/2})\e^{\|\phi\|_{*,\xi}}\int\limits_{B_1(\xi)\setminus\D_{\lambda,\xi}}|x-\xi|^{-\sigma} \,\mathrm{d}x\\
&\leq C(\lambda+\lambda^{\sigma/2})\e^{\lambda^{1/2}}.
\end{align*}
Likewise, \eqref{39} and \eqref{47a} yield
\begin{align*}
\int\limits_{\R^3\setminus B_1(\xi)}\left|M(\phi_{\lambda,\xi})\dfrac{\partial U_\alpha}{\partial x_i}\right|\,\mathrm{d}x&\leq C(\|\phi\|_{*,\xi}^2+\lambda)\e^{\|\phi\|_{*,\xi}}\int\limits_{\R^3\setminus B_1(\xi)}|x-\xi|^{-3-\beta} \,\mathrm{d}x\\
&\leq C\lambda\e^{\lambda^{1/2}},
\end{align*}
These inequalities prove \eqref{49b}.

On the other hand, we need to estimate the boundary integral of \eqref{49}. We claim that
\begin{equation}\label{50}
\left|\dfrac{\partial\phi_{\lambda,\xi}}{\partial n}(x)\right|=O(\lambda^{-1/2}),\quad\text{uniformly for}\ x\in\partial\D_{\lambda,\xi}.
\end{equation}
In fact, 
\[
 \widetilde{\phi}_{\lambda,\xi}(y)=\phi_{\lambda,\xi}\left(\xi+\sqrt{\dfrac{\lambda }{\lambda_0}}y\right)\quad\text{for all}\ y\in\R^3\setminus\D.
\]
By \eqref{47a}, we have
\[
|\widetilde{\phi}_{\lambda,\xi}(y)|\leq C\lambda^{1/2}\quad\text{for all}\ |y|\leq\sqrt{\dfrac{\lambda_0}{\lambda}}.
\]
Likewise, by \eqref{47a} and the definition of the norm $\|\|_{*,\xi}$, we see that $\phi_{\lambda,\xi}$ is uniformly bounded. Furthermore, \eqref{43} implies that $|\Delta\phi_{\lambda,\xi}|\leq C$ in $\R^N\setminus\D_{\lambda,\xi}$, thereby $\widetilde{\phi}_{\lambda,\xi}$ satisfies \[|\Delta\widetilde{\phi}_{\lambda,\xi}|\leq C\lambda.\]
By elliptic estimates \[\sup_{\partial\D}|\nabla\widetilde{\phi}_{\lambda,\xi}|\leq C\lambda^{1/2},\]
which proves \eqref{50}. Using the last inequality we derive
\[
\int\limits_{\partial\D_{\lambda,\xi}}\dfrac{\partial\phi_{\lambda,\xi}}{\partial n}\dfrac{\partial U_\alpha}{\partial x_i}\,\mathrm{dS}(x)=O(\lambda).
\]
This fact together \eqref{49b} prove the claim made in \eqref{49a}.

Finally, let us consider the vector field 
\[
G(\xi)=(G_1(\xi),G_2(\xi),G_3(\xi)).
\]
$G$ is continuous and, thanks to \eqref{49a}, 
\[
G(\xi)\cdot\xi<0\quad\text{for all } |\xi|=R,
\]
for any fixed small $R>0$. Using this and degree theory we obtain the existence of $\xi$ such that $c_1=c_2=c_3=0$. Which is the desired conclusion.
\qed

\bigskip
\noindent
{\bf Acknowledgment.} J.D. was supported by  Fondecyt  1090167, CAPDE-Anillo ACT-125 and
Fondo Basal CMM. L.L. was supported by a doctoral grant of CONICYT Chile.


\providecommand{\bysame}{\leavevmode\hbox to3em{\hrulefill}\thinspace}
\providecommand{\MR}{\relax\ifhmode\unskip\space\fi MR }
\providecommand{\MRhref}[2]{%
  \href{http://www.ams.org/mathscinet-getitem?mr=#1}{#2}
}
\providecommand{\href}[2]{#2}

\end{document}